\documentclass[12pt,a4paper,twoside]{article}

\title{\sc Functional A Posteriori Error Estimates for Elliptic Problems in Exterior Domains}
\def\shorttitle{A Posteriori Error Estimates in Exterior Domains}
\def\pauthor{Dirk Pauly and Sergey Repin}

\def\mylabelonoff{off}
\def\allowdisbrk{no}

\usepackage{a4,exscale,ifthen,amsfonts,amssymb,amsmath,amscd,graphicx}
\usepackage[english]{babel}
\usepackage[mathscr]{eucal}

\author{{\sf\pauthor}}
\numberwithin{equation}{section}

\newenvironment{acknow}{{\vspace*{1cm}\noindent\bf Acknowledgements }}{}
\newcommand{\bewboxw}{\mbox{}\hfill $\square$ \\}

\newenvironment{proof}{{\noindent\bf Proof }}{\bewboxw}

\newcommand{\keywords}[1]{{\noindent\bf Key Words }#1}
\newcommand{\amsclass}[1]{{\noindent\bf AMS MSC-Classifications }#1}

\ifthenelse{\equal{\mylabelonoff}{on}}
{\newcommand{\mylabel}[1]{\label{#1}\fbox{{\rm #1}}}}{\newcommand{\mylabel}[1]{\label{#1}\makebox[0mm][]{}}}
\ifthenelse{\equal{\allowdisbrk}{yes}}
{\allowdisplaybreaks}{}

\newcommand{\paper}[7]{\bibitem{#1} #2, `#3', {\it #4}, #5, (#6), #7.}

\newcommand{\book}[6]{\bibitem{#1} #2, {\it #3}, #4, #5, (#6).}

\usepackage[mathscr]{eucal}

\newcommand{\rz}{\mathbb{R}}
\newcommand{\rzp}{\rz_{+}}

\newcommand{\rzpm}{\rz_{\pm}}
\newcommand{\rN}{\rz^N}

\newcommand{\dd}{\text{d}}
\newcommand{\intl}{\,\dd\lambda}
\newcommand{\ints}{\,\dd\sigma}
\DeclareMathOperator{\p}{\partial}

\def\na{\nabla}

\renewcommand{\div}{\operatorname{div}}

\newcommand{\eps}{\varepsilon}

\newcommand{\om}{\Omega}
\newcommand{\Omegai}{\om_i}
\newcommand{\Omegae}{\om_e}
\newcommand{\omi}{\om_i}
\newcommand{\ome}{\om_e}

\def\omeps{\om\setminus B_{\eps}}
\newcommand{\trg}{\tau_\gamma}
\newcommand{\trG}{\tau_\Gamma}
\newcommand{\ntrg}{\tau_{n,\gamma}}
\newcommand{\ntrG}{\tau_{n,\Gamma}}

\def\intO{\int_{\om}}

\def\ye{y_e}
\def\yi{y_i}

\def\ug{u_{\gamma}}

\def\uge{\hat{u}}
\def\vge{\hat{v}}

\def\Set#1#2{\left\{#1\,\mid\,#2\right\}}

\DeclareMathOperator{\Lebesgue}{\mathsf{L}}
\newcommand{\Lgen}[2]{\Lebesgue^{#1}_{#2}}
\def\Lo{\Lgen{1}{}}

\def\Lou{\Lo(U_{1})}
\def\Lt{\Lgen{2}{}}
\def\Lts{\Lgen{2}{s}}

\def\Lto{\Lgen{2}{1}}
\def\Ltmo{\Lgen{2}{-1}}
\def\Ltom{\Lt(\om)}
\def\Ltsom{\Lts(\om)}

\def\Ltoom{\Lto(\om)}
\def\Ltmoom{\Ltmo(\om)}

\DeclareMathOperator{\Sobolev}{\mathsf{H}}
\newcommand{\Hgen}[3]{\overset{#3}{\Sobolev}{}^{#1}_{#2}}
\def\Homo{\Hgen{1}{-1}{}}
\def\Homoz{\Hgen{1}{-1}{\circ}}
\def\Hoot{\Hgen{1/2}{}{}}
\def\Hmoot{\Hgen{-1/2}{}{}}
\def\HootG{\Hoot(\Gamma)}
\def\HmootG{\Hmoot(\Gamma)}
\def\Homoom{\Homo(\om)}

\def\Homozom{\Homoz(\om)}

\def\Ho{\Hgen{1}{}{}}

\def\Hoomi{\Ho(\omi)}

\DeclareMathOperator{\Cont}{\mathsf{C}}
\newcommand{\Cgen}[2]{\overset{#2}{\Cont}{}^{#1}}

\def\Ciz{\Cgen{\infty}{\circ}}

\def\Cizom{\Ciz(\om)}

\DeclareMathOperator{\Divergence}{\mathsf{D}}
\newcommand{\Dgen}[1]{\Divergence_{#1}}
\def\D{\Dgen{}}
\def\Dom{\D(\om)}
\def\Domi{\D(\omi)}
\def\Dome{\D(\ome)}

\newcommand{\normdst}{\hspace{-0.4ex}}
\newcommand{\scp}[2]{\left\langle#1,#2\right\rangle}
\newcommand{\scpom}[2]{\scp{#1}{#2}_{\om}}

\newcommand{\scpomeps}[2]{\scp{#1}{#2}_{\omeps}}
\newcommand{\scpsom}[2]{\scp{#1}{#2}_{s,\om}}
\newcommand{\scpoom}[2]{\scp{#1}{#2}_{1,\om}}
\newcommand{\scpmoom}[2]{\scp{#1}{#2}_{-1,\om}}
\newcommand{\scpomi}[2]{\scp{#1}{#2}_{\omi}}
\newcommand{\scpome}[2]{\scp{#1}{#2}_{\ome}}
\newcommand{\scpg}[2]{\scp{#1}{#2}_{\gamma}}
\newcommand{\scpG}[2]{\scp{#1}{#2}_{\Gamma}}
\newcommand{\norm}[1]{\left|\normdst\left|#1\right|\normdst\right|}
\newcommand{\normom}[1]{\norm{#1}_{\om}}

\newcommand{\normomeps}[1]{\norm{#1}_{\omeps}}
\newcommand{\normsom}[1]{\norm{#1}_{s,\om}}
\newcommand{\normoom}[1]{\norm{#1}_{1,\om}}
\newcommand{\normmoom}[1]{\norm{#1}_{-1,\om}}
\newcommand{\normAom}[1]{\norm{#1}_{A,\om}}
\newcommand{\normAmoom}[1]{\norm{#1}_{A^{-1},\om}}
\newcommand{\normomi}[1]{\norm{#1}_{\omi}}
\newcommand{\normoome}[1]{\norm{#1}_{1,\ome}}
\newcommand{\normoomi}[1]{\norm{#1}_{1,\omi}}
\newcommand{\normmoome}[1]{\norm{#1}_{-1,\ome}}

\newtheorem{lem}{Lemma}

\newtheorem{theo}[lem]{Theorem}
\newtheorem{cor}[lem]{Corollary}
\newtheorem{rem}[lem]{Remark}
\newtheorem{pro}[lem]{Proposition}

\begin{document}

\maketitle{}

\begin{abstract}
This paper is concerned with the derivation of computable and guaranteed upper bounds
of the difference between the exact and the approximate solution
of an exterior domain boundary value problem for a linear elliptic equation.
Our analysis is based upon purely functional argumentation
and does not attract specific properties of an approximation method.
Therefore, the estimates derived in the paper at hand are applicable
to any approximate solution that belongs to the corresponding energy space.
Such estimates (also called error majorants of the functional type)
have been derived earlier for problems in bounded domains of $\rN$
(see \cite{repinconvex,RdeGruyter}).\\
\\
\amsclass{65 N 15}
\end{abstract}

\tableofcontents

\section{Introduction}

The main focus of our investigations is to suggest a method of deriving guaranteed and computable
upper bounds of the difference between the exact solution $u$
of an elliptic exterior domain boundary value problem and any approximation
from the corresponding energy space. We discuss the method with the paradigm
of the prototypical elliptic problem
\begin{align}
-\div A\na u&=f&&\text{in }\Omega,\mylabel{elleq}\\
u|_\gamma&=g&&\text{on }\gamma:=\p\Omega.\mylabel{bdcond}
\end{align}
We assume that $\Omega\subset\rN$ with $N\geq1$
is an exterior domain, i.e. $\rN\setminus\Omega$ is compact,
with Lipschitz continuous boundary $\gamma$
(see Figure \ref{domain}).

\begin{figure}
\centerline{\includegraphics[height=3in]{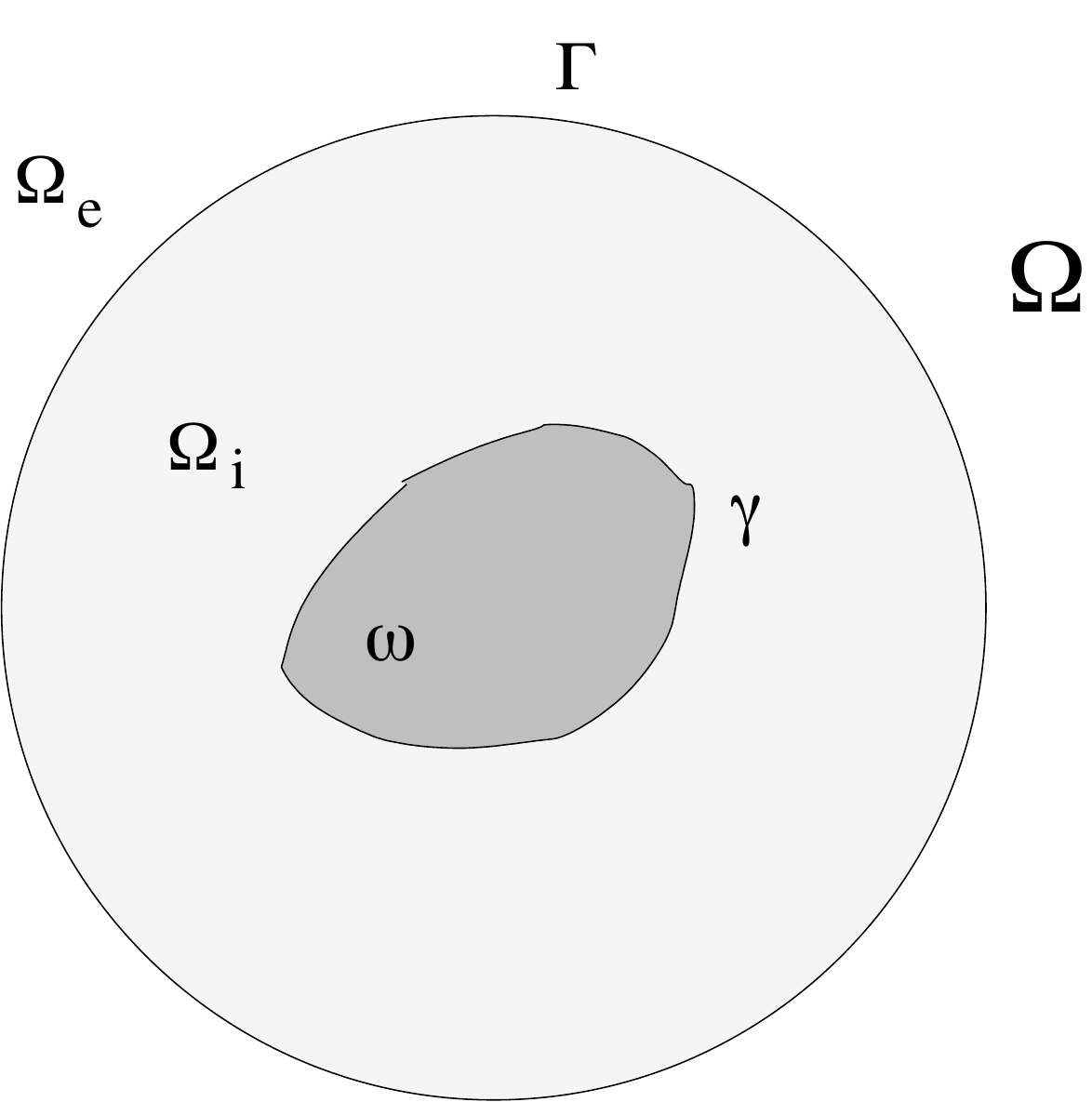}}
\mylabel{domain}
\caption{exterior domain $\Omega$ with artificial interface $\Gamma$}
\end{figure}

Throughout this paper we will use the weighted Lebesgue function spaces
$$\Ltsom:=\Set{\varphi}{\rho^s\varphi\in\Ltom},\quad s\in\rz.$$
Here $\rho:=(1+r^2)^{1/2}$ and $r(x):=|x|$ denotes the radius vector.
$\Ltsom$ is a Hilbert space equipped with the scalar product
$$\scpsom{\varphi}{\psi}:=\scpom{\rho^{s}\varphi}{\rho^{s}\psi}
:=\intO\rho^{2s}\varphi\psi\intl,$$
where $\varphi$ and $\psi$ belong to $\Ltsom$ and $\lambda$ is Lebesgue's measure.
We denote the corresponding norms by $\normsom{\varphi}=\normom{\rho^{s}\varphi}$.
If $s=0$ then $\Ltsom$ coincides with the usual Lebesgue space $\Ltom$.
For the sake of simplicity we keep the same notation
for spaces of vector-valued functions.
Moreover, we introduce the weighted Sobolev space
$$\Homoom:=\Set{\varphi\in\Ltmoom}{\na\varphi\in\Ltom},$$
which is a Hilbert space as well with respect to the scalar product
$$(\varphi,\psi)\mapsto\scpmoom{\varphi}{\psi}+\scpom{\na\varphi}{\na\psi}.$$
By $\Homozom$ we denote the closure of $\Cizom$,
the space of compactly supported smooth test functions,
in the norm of $\Homoom$.
Whenever we consider Sobolev spaces of bounded domains
we use the usual unweighted $\Lt$-scalar products and -norms.

For dimensions $N\geq3$ the solution theory for the problem \eqref{elleq}-\eqref{bdcond}
is based on the weighted Poincare/Friedrich estimate
(see Corollary \ref{poincareestcor} (i) and Remark \ref{poincareestrem} of the appendix)
\begin{align}
\normmoom{\varphi}\leq\frac{2}{N-2}\normom{\na\varphi}\quad
\forall\,\varphi\in\Homozom,\mylabel{Poincare}
\end{align}
the Lax-Milgram theorem and, if needed, an adequate extension operator for the boundary data.
Let $\ug$ be some function in $\Homoom$ satisfying the boundary condition \eqref{bdcond}.
The weak solution $u\in\Homozom+\ug\subset\Homoom$ of \eqref{elleq}-\eqref{bdcond} 
is then defined by the variational formulation
\begin{align}
\scpom{A\na u}{\na w}&=\scpom{f}{w}\quad\forall\,w\in\Homozom.\mylabel{varformu}
\end{align}
By \eqref{Poincare} the left hand side of \eqref{varformu}
is a strongly coercitive sesqui-linear form over $\Homozom$
provided that the real-matrix-valued function $A$ is measurable, bounded a.e.,
symmetric and uniformly strongly elliptic, i.e.
\begin{align}
\exists\,c_A>0\quad\forall\,\xi\in\rN\quad\forall\,x\in\Omega\quad
c_A|\xi|^2\leq A(x)\xi\cdot\xi.\mylabel{cA}
\end{align}
If $f\in\Ltoom$ then by the Cauchy-Scharz inequality the right hand side of \eqref{varformu}
is a linear and continuous functional over $\Homozom$.
Thus, under these assumptions the problem \eqref{varformu}
is uniquely solvable in $\Homozom+\ug$ by Lax-Milgram's theorem.

If $N=1,2$ one can apply the same arguments with the difference
that \eqref{Poincare} has to be modified.
For $N=1$ and, for example, $\Omega\subset\rzp$
we have by Corollary \ref{poincareestcor} (iii) and Remark \ref{poincareestrem}
\begin{align}
\normmoom{\varphi}\leq2\normom{\varphi'}
\quad\forall\,\varphi\in\Homozom.\mylabel{poincareonedim}
\end{align}
Hence, we get the same solution theory with tiny restrictions on $\Omega$,
which easily can be removed by a translation.
For $N=2$ the singularities are stronger
and additionally we have to utilize logarithmic terms.
By Corollary \ref{poincareestcor} (ii) and Remark \ref{poincareestrem} we have
for domains $\Omega\subset\rz^2$,
such that the complement $\rz^2\setminus\Omega$ contains the unit ball,
\begin{align}
\normom{\varphi/(r\ln r)}\leq2\normom{\na\varphi}
\quad\forall\,\varphi\in\Hgen{1}{-1,\ln}{\circ}(\Omega),\mylabel{poincaretwoim}
\end{align}
where
$$\Hgen{1}{-1,\ln}{}(\Omega):=\Set{\varphi}{\varphi/(r\ln r),\na\varphi\in\Ltom}$$
is a Hilbert space equipped with the natural scalar product
$$(\varphi,\psi)\mapsto\scpom{\varphi/(r\ln r)}{\psi/(r\ln r)}+\scpom{\na\varphi}{\na\psi}$$
and again $\Hgen{1}{-1,\ln}{\circ}(\Omega)$ denotes the closure of $\Cizom$
in the norm of $\Hgen{1}{-1,\ln}{}(\Omega)$.
Consequently, we obtain for all $f$ with $r\ln rf\in\Ltom$
and all $\ug$ in $\Hgen{1}{-1,\ln}{}(\Omega)$ satisfying the boundary condition \eqref{bdcond}
a unique solution $u$ belonging to $\Hgen{1}{-1,\ln}{\circ}(\Omega)+\ug$.

We summarize the results in the following

\begin{theo}
\mylabel{solutiontheo}
Let $N\geq3$ as well as $f\in\Ltoom$ and
$\ug\in\Homoom$ satisfying the boundary condition \eqref{bdcond}.
Then the exterior boundary value problem \eqref{elleq}-\eqref{bdcond}
is uniquely weakly solvable in $\Homozom+\ug$. The solution operator is continuous.
\end{theo}

From the above discussion, it is clear that for $N=1,2$ the existence
of weak solutions in suitable spaces can also  be proved.

\begin{rem}
\mylabel{traceextension}
The boundary data $g$ and its extension $\ug$ can be described in more detail.
In the bounded domain case it is well known that there exists a bounded
linear trace operator and a corresponding bounded linear extension operator (right inverse)
mapping $\Hgen{1}{}{}(\Omega)$ to $\Hgen{1/2}{}{}(\gamma)$ and vice verse.
Hence, by restriction we get a bounded  linear trace operator
$$\trg:\Homoom\to\Hoot(\gamma)$$
and by extension and applying an obvious cutting technique we obtain a bounded linear extension operator
$$E:\Hgen{1/2}{}{}(\gamma)\to\Homoom$$
for our exterior domain $\Omega$, which even maps to functions with (arbitrarily thin) compact support.
As in the bounded domain case, $E$ is a right inverse of $\trg$.
Then we may specify $g\in\Hgen{1/2}{}{}(\gamma)$ and $\ug:=Eg\in\Homoom$
as well as our variational formulation for $u=\tilde{u}+Eg$: Find $\tilde{u}\in\Homozom$, such that
$$B(\tilde{u},w):=\scpom{A\na\tilde{u}}{\na w}=\scpom{f}{w}-\scpom{A\na Eg}{\na w}=:F(w)
\quad\forall\,w\in\Homozom.$$
\end{rem}

Finally, we introduce
$$\Dom:=\Set{\varphi\in\Ltom}{\div\varphi\in\Ltoom},$$
which is a Hilbert space with respect to the canonical scalar product
$$(\varphi,\psi)\mapsto\scpom{\varphi}{\psi}+\scpoom{\div\varphi}{\div\psi}.$$

\section{Upper bounds for the deviation from the exact solution in dimensions $N\geq3$}

Let $v$ be an approximation of $u\in\Homozom+\ug\subset\Homoom$,
where $v$ is assumed just to belong to $\Homoom$
since the boundary condition may not be satisfied exactly.
Our goal is to obtain upper bounds for the difference
between $\na u$ and $\na v$ in terms of the norm
$$\normAom{\varphi}:=\normom{A^{1/2}\varphi}=\scpom{A\varphi}{\varphi}^{1/2}.$$
We use \eqref{varformu} and get for all $w\in\Homozom$
\begin{align}
\scpom{A\na(u-v)}{\na w}&=\scpom{f}{w}-\scpom{A\na v}{\na w}.\mylabel{stepone}
\end{align}

Before we proceed we note two useful results.

\begin{theo}
\mylabel{estimatetheo}
Let $u,v\in\Homoom$ be as above.
Moreover, let $\Phi$ be a linear and continuous functional over $\Homozom$
and $c_\Phi>0$, such that for all $w\in\Homozom$
$$\scpom{A\na(u-v)}{\na w}=\Phi(w)\leq c_\Phi\normAom{\na w}$$
holds. Then
\begin{align}
\normAom{\na(u-v)}&\leq c_\Phi+2\normAom{\na(\uge-\vge)}\mylabel{estimatetheoestone}
\intertext{for all $\uge,\vge\in\Homoom$, 
for which $\uge-\vge$ coincides with $u-v$ on the boundary $\gamma$.
If additionally $u-v$ belongs to $\Homozom$ then}
\normAom{\na(u-v)}&\leq c_\Phi.\mylabel{estimatetheoesttwo}
\end{align}
\end{theo}

\begin{proof}
We consider
$$w:=u-v-(\uge-\vge)\in\Homozom.$$
Using Cauchy-Schwarz' inequality we obtain
\begin{align*}
\normAom{\na w}^2
&=\scpom{A\na(u-v)}{\na w}-\scpom{A\na(\uge-\vge)}{\na w}\\
&\leq\left(c_\Phi+\normAom{\na(\uge-\vge)}\right)\normAom{\na w}
\end{align*}
and thus $\normAom{\na w}\leq c_\Phi+\normAom{\na(\uge-\vge)}$.
By the triangle inequality we get \eqref{estimatetheoestone}.
\eqref{estimatetheoesttwo} is trivial since we can set $w:=u-v$,
i.e. $\uge:=\vge:=0$.
\end{proof}

We may be more specific using the trace and extension operators from
Remark \ref{traceextension}.

\begin{cor}
\mylabel{estimatecor}
Let the assumptions of Theorem \ref{estimatetheo} be satisfied. Then
$$\normAom{\na(u-v)}
\leq c_\Phi+2\normAom{\na E(g-\trg v)}
\leq c_\Phi+2c_\gamma\norm{g-\trg v}_{\Hoot(\gamma)}.$$ 
Here $c_\gamma>0$ is the constant in the inequality
\begin{align}
\normAom{\na E\varphi}\leq c_\gamma\norm{\varphi}_{\Hoot(\gamma)}
\quad\forall\varphi\in\Hoot(\gamma).\mylabel{extensionineq}
\end{align}
\end{cor}

\begin{proof} 
Setting $\uge:=Eg$ and $\vge:=E\trg v$ as well as using
\eqref{extensionineq} proves the inequalities.
We note that \eqref{estimatetheoesttwo} follows directly from the corollary as well.
\end{proof}

In the subsequent sections we introduce and discuss some different
functionals $\Phi$ and corresponding constants $c_\Phi$.

\subsection{First estimate}

For any $y\in\Dom$ and any $w\in\Homozom$ we have
\begin{align}
\scpom{\div y}{w}+\scpom{y}{\na w}&=0.\mylabel{steptwo}
\end{align}
Combining \eqref{stepone} and \eqref{steptwo} we obtain for all $w\in\Homozom$ and all $y\in\Dom$
\begin{align}
\scpom{A\na(u-v)}{\na w}&=\scpom{f+\div y}{w}+\scpom{y-A\na v}{\na w}=:\Phi(w).\mylabel{stepthree}
\end{align}
By Cauchy-Schwarz' inequality, \eqref{Poincare} with $c_N:=2/(N-2)$ and \eqref{cA}
we estimate the right hand side $\Phi(w)$ of \eqref{stepthree} as follows:
\begin{align}
\left|\scpom{f+\div y}{w}\right|&\leq\normoom{f+\div y}\normmoom{w}
\leq c_N\normoom{f+\div y}\normom{\na w}\mylabel{estimateone}\\
&\leq\frac{c_N}{\sqrt{c_A}}\normoom{f+\div y}\normAom{\na w}\nonumber\\
\left|\scpom{y-A\na v}{\na w}\right|&\leq\normAmoom{y-A\na v}\normAom{\na w}\mylabel{estimatetwo}
\end{align}
By Corollary \ref{estimatecor} we arrive at the following result.

\begin{pro}
\mylabel{estone}
Let $u,v$ be as in Theorem \ref{estimatetheo}. Then
\begin{equation}
\mylabel{estproone}
\normAom{\na(u-v)}\leq\frac{c_N}{\sqrt{c_A}}\normoom{f+\div y}+\normAmoom{y-A\na v}
+2c_\gamma\norm{g-\trg v}_{\Hoot(\gamma)},
\end{equation}
where $y$ is an arbitrary vector field in $\Dom$.
\end{pro}

\begin{rem}
\mylabel{estonerem}
If $v$ satisfies the prescribed boundary condition, then \eqref{estproone} implies
\begin{equation}
\mylabel{estproonem}
\normAom{\na(u-v)}\leq\frac{c_N}{\sqrt{c_A}}\normoom{f+\div y}+\normAmoom{y-A\na v}.
\end{equation}
The estimates \eqref{estproone} and \eqref{estproonem}) show that deviations from 
exact solutions of exterior boundary value problems have the same structure
as for problems in bounded domains, namely they contain weighted residuals
of basic relations with weights given by constants in the corresponding
embedding inequalities.
\end{rem}

\subsection{Second estimate}\mylabel{secondform}

Assume that $\Omega$ is decomposed into two subdomains $\Omegai$ and $\Omegae$
with interface $\Gamma:=\p\Omegae$ (see Figure \ref{domain})
and that the fields $y\in\Dom$ exactly satisfy the relation
\begin{align}
\div y+f=0\quad\text{in }\Omegae.\mylabel{divyome}
\end{align}
In particular, this situation may arise if the source term $f$ has compact support
and $y$ is represented (in the exterior domain $\Omegae$)
as a linear combination of solenoidal fields having proper decay at infinity.
In this case, the estimate of Proposition \ref{estone} turns trivially to
\begin{align}
\normAom{\na(u-v)}\leq c_o\normomi{f+\div y}+\normAmoom{y-A\na v}
+2c_\gamma\norm{g-\trg v}_{\Hoot(\gamma)},\mylabel{aposterioriestimatetwo}
\end{align}
which holds for all $y\in\Dom$ additionally satisfying \eqref{divyome},
where the weight constant is
\begin{equation}
\mylabel{co}
c_o:=\frac{c_N(1+\norm{r}_{\infty,\Omegai})}{\sqrt{c_A}},
\end{equation}
which follows directly from
$$\normoom{f+\div y}=\normoomi{f+\div y}\leq|\rho|_{\infty,\Omegai}\normomi{f+\div y}
\leq(1+|r|_{\infty,\Omegai})\normomi{f+\div y}.$$
But we also may derive another estimate. We rewrite \eqref{estimateone} 
and use Cauchy-Schwarz' inequality in $\Omegai$
\begin{equation}
\left|\scpom{f+\div y}{w}\right|
=\left|\scpomi{f+\div y}{w}\right|\leq\normomi{f+\div y}\normomi{w}
\mylabel{estimatethreeone}
\end{equation}
and estimate
\begin{equation}
\normomi{w}\leq c_{\Omegai}\normomi{\na w}
\leq\frac{c_{\Omegai}}{\sqrt{c_A}}\normAom{\na w}.
\mylabel{estimatethreetwo}
\end{equation}
Here $c_{\Omegai}$ denotes a Poincare/Friedrich constant
associated with the bounded domain $\Omegai$, i.e. the best constant of the inequality
\begin{align*}
\normomi{\varphi}\leq c_{\Omegai}\normomi{\na\varphi}
\quad\forall\,\varphi\in\Set{\psi\in\Hoomi}
{\tau_{\p\Omegai}\psi|_\gamma=0\text{ on }\gamma},
\end{align*}
where $\tau_{\p\Omegai}:\Hoomi\to\Hoot(\p\Omegai)$ denotes the trace operator.
In this case, we have again \eqref{aposterioriestimatetwo}
but now with the (optional) weight constant
\begin{equation}
\mylabel{coo}
c_o:=\frac{c_{\Omegai}}{\sqrt{c_A}}.
\end{equation}
We note that the constant \eqref{co} may also be achieved
by \eqref{estimateone} and the argument \eqref{estimatethreeone}
if we replace the estimate \eqref{estimatethreetwo} by
\begin{align*}
\normomi{w}&\leq(1+|r|_{\infty,\Omegai})\norm{w}_{-1,\Omegai}
\leq(1+|r|_{\infty,\Omegai})\normmoom{w}
\leq\frac{c_N}{\sqrt{c_A}}(1+|r|_{\infty,\Omegai})\normAom{\na w}.
\end{align*}

We summarize and get our second a posteriori error estimate.

\begin{pro}
\mylabel{esttwo}
For all $y\in\Dom$ with \eqref{divyome} we have
$$\normAom{\na(u-v)}\leq c_o\normomi{f+\div y}+\normAmoom{y-A\na v}
+2c_\gamma\norm{g-\trg v}_{\Hoot(\gamma)},$$
where $c_{o}$ is defined either by \eqref{co} or by \eqref{coo}.
\end{pro}

\begin{rem}
\mylabel{esttworem}
In general,  the number $c_{\Omegai}$ will be smaller 
and thus provides a better bound
than $c_N(1+\norm{r}_{\infty,\Omegai})$.
On the other hand, the number $c_N(1+\norm{r}_{\infty,\Omegai})/\sqrt{c_A}$
is an easily computable upper bound for 
the best possible constant $c_{o}$.
\end{rem}

\subsection{Third estimate}

Let $\yi$ and $\ye$ be the restrictions of some
$y\in\Ltom$ to $\Omegai$ and $\Omegae$, respectively.
Assuming $\yi\in\Domi$ and $\ye\in\Dome$ 
but not necessarily $y\in\Dom$ we use the equations
\begin{align}
\scpomi{\yi}{\na w}+\scpomi{\div\yi}{w}&=\scpG{\ntrG\yi}{\trG w},\mylabel{traceone}\\
\scpome{\ye}{\na w}+\scpome{\div\ye}{w}&=-\scpG{\ntrG\ye}{\trG w},\mylabel{tracetwo}
\end{align}
which hold for all $w\in\Homozom$ and in the sense of the traces
$\trG:\Homoom\to\HootG$ and
$\ntrG:\Domi\to\HmootG$ respectively $\ntrG:\Dome\to\HmootG$.
At this point we  assume that the interface $\Gamma$
is Lipschitz (in order to guarantee that the traces are well defined).
By $\scpG{\varphi}{\psi}$ we denote the duality product of
$\HmootG$ and $\HootG$.
We recall that the normal traces $\ntrG\yi$ and $\ntrG\ye$ 
possess weak surface divergences in $\HmootG$ as well.
If $y\in\Dom$, then $\div\yi=\div y$ in $\Omegai$
and $\div\ye=\div y$ in $\Omegae$.
Hence, in this case adding \eqref{traceone} and \eqref{tracetwo}
we obtain by \eqref{steptwo}
$$\scpG{\ntrG\yi-\ntrG\ye}{\trG w}
=\scpomi{y}{\na w}+\scpom{\div y}{w}=0$$
for all $w\in\Homozom$. Therefore, we get
$$\ntrG\yi=\ntrG\ye$$
for all $y\in\Dom$ since $\trG$ is surjective.

On our way to find $\Phi$ like in \eqref{stepthree}
we now insert \eqref{traceone}, \eqref{tracetwo}
instead of \eqref{steptwo} into \eqref{stepone} and obtain
\begin{align}
\scpom{A\na(u-v)}{\na w}
&=\scpomi{f+\div\yi}{w}+\scpome{f+\div\ye}{w}\mylabel{stepthreemod}\\
&\qquad+\scpom{y-A\na v}{\na w}+\scpG{\ntrG\ye-\ntrG\yi}{\trG w}=:\Phi(w).\nonumber
\end{align}

The third term of $\Phi(w)$ will be estimated by \eqref{estimatetwo}
and for the last term we may use the continuity of the trace operator $\trG$
in combination with a Poincare/Friedrich estimate, i.e.
\begin{align}
\norm{\trG\varphi}_{\HootG}
\leq c_\Gamma\normAom{\na\varphi}\quad\forall\varphi\in\Homozom,\mylabel{traceineq}
\end{align}
and obtain
\begin{align}
\left|\scpG{\ntrG\ye-\ntrG\yi}{\trg w}\right|
&\leq\norm{\ntrG\ye-\ntrG\yi}_{\HmootG}\norm{\trG w}_{\HootG}\mylabel{estimatefour}\\
&\leq c_\Gamma\norm{\ntrG\ye-\ntrG\yi}_{\HmootG}\normAom{\na w}.\nonumber
\end{align}

To estimate the second term of $\Phi(w)$ we again
use \eqref{Poincare} and \eqref{cA} and obtain
\begin{align}
\left|\scpome{f+\div\ye}{w}\right|
&\leq\normoome{f+\div\ye}\normmoome{w}\leq\normoome{f+\div\ye}\normmoom{w}\mylabel{estimatefive}\\
&\leq\frac{c_N}{\sqrt{c_A}}\normoome{f+\div\ye}\normAom{\na w}.\nonumber
\end{align}

Considering the first (and last) term of $\Phi(w)$
we have once more at least two options 
as in section \ref{secondform} to obtain the estimate
\begin{align}
\left|\scpomi{f+\div\yi}{w}\right|
&\leq c_o\normomi{f+\div\yi}\normAom{\na w}\mylabel{estimatesix}
\end{align}
with $c_o$ defined either by \eqref{co} or \eqref{coo}.

Finally with \eqref{stepthreemod}
and \eqref{estimatetwo}, \eqref{estimatefour}, \eqref{estimatefive}, \eqref{estimatesix}
we get by Corollary \ref{estimatecor} the third estimate.

\begin{pro}
\mylabel{estthree}
For all $y\in\Ltom$ with $\yi\in\Domi$ and $\ye\in\Dome$ we have
\begin{align}
\mylabel{estimatethree}
\normAom{\na(u-v)}&\leq c_o\normomi{f+\div\yi}
+\frac{c_N}{\sqrt{c_A}}\normoome{f+\div\ye}
+\normAmoom{y-A\na v}\\
&\qquad+c_\Gamma\norm{\ntrG\ye-\ntrG\yi}_{\HmootG}
+2c_\gamma\norm{g-\trg v}_{\Hoot(\gamma)}\nonumber
\end{align}
with $c_{o}$ from Proposition \ref{esttwo}. The right hand side of
\eqref{estimatethree} vanishes if and only if $v$ coincides with $u$
and $y$ with $A\na u$.
\end{pro}

\begin{rem}
\mylabel{estthreeremone}
There are many ways to deduce \eqref{traceineq}.
We just mention that $\trG\varphi$ can be considered
as a trace of a function defined in $\Omegai$ or $\Omegae$
or even of a function, 
which is just defined in a small neighborhood of $\Gamma$.
Thus, we may adjust the constant $c_\Gamma$ according to our needs.
\end{rem}

\begin{rem}
\mylabel{estthreeremtwo}
This estimate suggests even a solution method:
We construct approximations using locally supported trial
functions in $\Omegai$, e.g. FEM, and utilize global approximations
properly behaving at infinity for $\Omegae$. These two types
of approximations are usually difficult to meet together exactly
on the artificial boundary $\Gamma$.
However, Proposition \ref{estthree} shows that
this is not required because we can use instead the penalty term
with known penalty factor $c_\Gamma$. In addition, we have one
more parameter, the `radius' of the interface $\Gamma$.
Since $\Gamma$ is artificial and arbitrary
we can use this parameter in the algorithm
in order to obtain better results.
\end{rem}

\begin{rem}
\mylabel{estthreeremthree}
At this point we shall note that all our estimates are sharp,
which easily can be seen by setting $v:=u\in\Homoom$ and $y:=A\na u\in\Dom$.
\end{rem}

\begin{rem}
\mylabel{estthreeremfour}
In Propositions \ref{estone}, \ref{esttwo}, \ref{estthree}
we can always replace the last summand of the right hand side
by $2\normAom{\na(\uge-\vge)}$ or $2\normAom{\na E(g-\trg v)}$
using Theorem \ref{estimatetheo} and Corollary \ref{estimatecor}.
\end{rem}

\section{Upper bounds in dimension $N=2$}

Of course, Theorem \ref{estimatetheo} holds for $N=2$ as well and
the modifications on the estimates depend just on the Poincare/Friedrich estimate
and thus they are obvious using the proper Cauchy-Schwarz inequality.
We achieve

\begin{pro}
\mylabel{eststdim}
Let $\Omega\subset\rz^2$, such that $\rz^2\setminus\Omega$ contains the unit ball.
\begin{itemize}
\item[\bf(i)] For all $y\in\Dom$
$$\normAom{\na(u-v)}\leq\frac{2}{\sqrt{c_A}}\normom{r\ln r(f+\div y)}+\normAmoom{y-A\na v}
+2c_\gamma\norm{g-\trg v}_{\Hoot(\gamma)}.$$
\item[\bf(ii)] For all $y\in\Dom$ with $\div y+f=0$ in $\Omegae$
$$\normAom{\na(u-v)}\leq c_o\normomi{f+\div y}+\normAmoom{y-A\na v}
+2c_\gamma\norm{g-\trg v}_{\Hoot(\gamma)},$$
where $c_{o}=\min\left\{2\norm{r\ln r}_{\infty,\Omegai},c_{\Omegai}\right\}/\sqrt{c_A}$.
\item[\bf(iii)] For all $y\in\Ltom$ with $\yi\in\Domi$ and $\ye\in\Dome$
\begin{align*}
\normAom{\na(u-v)}&\leq c_o\normomi{f+\div\yi}
+\frac{2}{\sqrt{c_A}}\normoome{r\ln r(f+\div\ye)}
+\normAmoom{y-A\na v}\\
&\qquad+c_\Gamma\norm{\ntrG\ye-\ntrG\yi}_{\HmootG}
+2c_\gamma\norm{g-\trg v}_{\Hoot(\gamma)}.
\end{align*}
\end{itemize}
Analogously, Remarks \ref{esttworem}, \ref{estthreeremone}, \ref{estthreeremtwo},
\ref{estthreeremthree}, \ref{estthreeremfour} hold.
\end{pro}

\appendix

\section{Appendix}

\subsection{Lower bounds for the error}

We note by a standard variational argument
\begin{align*}
\normAom{\na(u-v)}^2
&=\sup_{y\in\Ltom}\left(2\scpom{A\na(u-v)}{y}-\normAom{y}^2\right).
\end{align*}
Thus, we obtain for all $w\in\Homoom$ the estimate
\begin{align*}
\normAom{\na(u-v)}^2
&\geq2\scpom{A\na(u-v)}{\na w}-\normAom{\na w}^2\\
&=2\scpom{A\na u}{\na w}-\scpom{A\na(2v+w)}{\na w},
\end{align*}
which is sharp since one can put $w=u-v$.
But to exclude the unknown exact solution $u$ from the right hand side
we need $w\in\Homozom$ since then by \eqref{varformu}
\begin{align}
\normAom{\na(u-v)}^2
&\geq2\scpom{f}{w}-\scpom{A\na(2v+w)}{\na w}.\mylabel{estbelow}
\end{align}
But this estimate is no longer sharp because we can not put $w=u-v$ anymore.
In fact, with $A\na u\in\Dom$ and $\div A\na u=-f$ we get for $w\in\Homoom$
$$\scpom{A\na u}{\na w}=\scpom{f}{w}+\scpg{\ntrg A\na u}{\trg w}.$$
Hence, we obtain the estimate
\begin{align*}
\normAom{\na(u-v)}^2
&\geq2\scpom{f}{w}-\scpom{A\na(2v+w)}{\na w}+2\scpg{\ntrg A\na u}{\trg w}
\end{align*}
for all $w\in\Homoom$, which is sharp and coincides with \eqref{estbelow}
if $w\in\Homozom$.
But the unknown exact solution $u$
still appears on the right-hand side,
i.e. the normal trace of $A\na u$ on $\gamma$.
Furthermore, if $\scpg{\ntrg A\na u}{\trg w}>0$ then \eqref{estbelow}
can not be sharp.

\subsection{Poincare type estimates for exterior domains}

We introduce the radial derivative $\p_{r}:=\xi\cdot\na$,
where $\xi(x):=x/r(x)$. Furthermore, $B_{\varepsilon}$ and $S_{\varepsilon}$
denote the open ball and sphere of radius $\varepsilon$ centered at the origin in $\rN$, respectively.
We will use the ideas of \cite[Lemma 4.1]{charlie}
and \cite[Poincare's estimate III, p. 57]{leisbook} with some minor useful modifications.

\begin{lem}
\mylabel{poincareestlemma}
Let $\Omega\subset\rN$, $N\geq1$, be a domain and $\beta\in\rz$.
For all $u\in\Cizom$ the following Poincare estimates hold:
\begin{itemize}
\item[\bf(i)] If $\beta>1-N/2$ then
$$(2\beta+N-2)\normom{r^{\beta-1}u}\leq2\normom{r^\beta\p_{r}u}.$$
\item[\bf(ii)] Let $B_1\subset\rN\setminus\Omega$.
If $\beta\geq(3-N)/2$ or $\beta\leq1-N/2$ then
$$|2\beta+N-3|\normom{\frac{r^{\beta-1}}{\ln r}u}\leq2\normom{r^\beta\p_{r}u}.$$
\item[\bf(iii)] If $N=1$ then
$$|2\beta-1|\normom{(1+r)^{\beta-1}u}\leq2\normom{(1+r)^\beta\p_{r}u}
+\left|2\min\{0,2\beta-1\}\right|^{1/2}|u(0)|,$$
where $u$ will be extended by zero to $\rz$.
\end{itemize}
\end{lem}

For the estimates derived in this paper it suffices to set $\beta=0$.
In this particular case, the above lemma implies

\begin{cor}
\mylabel{poincareestcor}
Let $\Omega\subset\rN$, $N\geq1$, be a domain.
For all $u\in\Cizom$ the following Poincare estimates hold:
\begin{itemize}
\item[\bf(i)] If $N\geq3$ then
$$\normmoom{u}\leq\normom{u/(1+r)}\leq\normom{u/r}
\leq\frac{2}{N-2}\normom{\p_{r}u}\leq\frac{2}{N-2}\normom{\na u}.$$
\item[\bf(ii)] If $N=2$ and $B_1\subset\rz^2\setminus\Omega$ then
$$\normom{u/(r\ln r)}\leq2\normom{\p_{r}u}\leq2\normom{\na u}.$$
\item[\bf(iii)] If $N=1$ then
$$\normmoom{u}\leq\normom{u/(1+r)}\leq2\normom{\p_{r}u}
+\sqrt2|u(0)|\leq2\normom{u'}+\sqrt2|u(0)|.$$
Hence, if $\Omega\subset\rzpm$ we have
$$\normmoom{u}\leq\normom{u/(1+r)}\leq2\normom{\p_{r}u}\leq2\normom{u'}.$$
\end{itemize}
\end{cor}

\begin{rem}
\mylabel{poincareestrem}
Of course, by continuity all these estimates extend to appropriate weighted $\Ho$-Sobolev spaces.
\end{rem}

\begin{proof}
Let $\Omega\subset\rN$, $N\geq1$, be a domain and $u\in\Cizom$.
By partial integration we get for all $\alpha\in\rz$ and $\varepsilon>0$
\begin{align*}
2\int_{\omeps}r^\alpha u\p_{r}u\intl
&=\int_{\omeps}r^\alpha\p_{r}|u|^2\intl\\
&=-(\alpha+N-1)\int_{\omeps}r^{\alpha-1}|u|^2\intl
-\varepsilon^{\alpha}\int_{S_{\varepsilon}}|u|^2\ints.
\end{align*}
Thus, for all $\gamma\in\rz$ and $\beta:=(\alpha+1)/2$
\begin{align*}
&\qquad\normomeps{r^\beta\p_{r}u+\gamma r^{\beta-1}u}^2\\
&=\normomeps{r^\beta\p_{r}u}^2+|\gamma|^2\normomeps{r^{\beta-1}u}^2
+2\gamma\underbrace{\scpomeps{r^\beta\p_{r}u}{r^{\beta-1}u}}
_{\displaystyle=\int_{\omeps}r^\alpha u\p_{r}u\intl}\\
&=\normomeps{r^\beta\p_{r}u}^2+\gamma(\gamma-2\beta-N+2)\normomeps{r^{\beta-1}u}^2
-\gamma\varepsilon^{2\beta-1}\int_{S_{\varepsilon}}|u|^2\ints.
\end{align*}
Now the left hand side of this equality converges by the monotone convergence theorem.
Since $r^\nu\in\Lou$, if and only if $\nu>-N$,
and $|\int_{S_{\varepsilon}}|u|^2\ints|\leq c\varepsilon^{N-1}$
the right hand side converges for $\beta>1-N/2$ by Lebesgue's dominated convergence theorem in $\rz$.
Hence, for $\varepsilon\to0$ we obtain
$$\normom{r^\beta\p_{r}u+\gamma r^{\beta-1}u}^2
=\normom{r^\beta\p_{r}u}^2+\gamma(\gamma-2\beta-N+2)\normom{r^{\beta-1}u}^2.$$
Choosing $\gamma:=2\beta+N-2>0$ we finally get by the triangle inequality
$$\gamma\normom{r^{\beta-1}u}\leq2\normom{r^\beta\p_{r}u}.$$
Since we are especially interested in the case $\beta=0$ this estimate is only applicable in dimensions $N\geq3$.

For $N=1$ we proceed as follows: For all $\alpha\in\rz$ we have
\begin{align*}
&\qquad2\int_{\rzpm}(1+r)^\alpha u\p_{r}u\intl
=\pm2\int_{\rzpm}(1\pm t)^\alpha u(t)u(t)'\,\dd t\\
&=\pm2\int_{\rzpm}(1\pm t)^\alpha\big(|u(t)|^2\big)'\,\dd t
=-\alpha\int_{\rzpm}(1\pm t)^{\alpha-1}|u(t)|^2\,\dd t-|u(0)|^2
\end{align*}
and thus
$$2\int_{\rz}(1+r)^\alpha u\p_{r}u\intl
=-\alpha\int_{\rz}(1+r)^{\alpha-1}|u(t)|^2\intl-2|u(0)|^2.$$
Hence, for all $\gamma\in\rz$ and $\beta:=(\alpha+1)/2$
\begin{align*}
&\qquad\normom{(1+r)^\beta\p_{r}u+\gamma(1+r)^{\beta-1}u}^2\\
&=\normom{(1+r)^\beta\p_{r}u}^2+|\gamma|^2\normom{(1+r)^{\beta-1}u}^2
+2\gamma\underbrace{\scpom{(1+r)^\beta\p_{r}u}{(1+r)^{\beta-1}u}}
_{\displaystyle=\intO(1+r)^\alpha u\p_{r}u\intl}\\
&=\normom{(1+r)^\beta\p_{r}u}^2
+\gamma(\gamma-2\beta+1)\normom{(1+r)^{\beta-1}u}^2-2\gamma|u(0)|^2.
\end{align*}
As before the triangle inequality and the choice $\gamma:=2\beta-1$,
but now without any restrictions on $\beta$, lead to
\begin{align*}
|\gamma|\normom{(1+r)^{\beta-1}u}
&\leq\normom{(1+r)^\beta\p_{r}u}
+\left(\normom{(1+r)^\beta\p_{r}u}^2-2\gamma|u(0)|^2\right)^{1/2}\\
&\leq2\normom{(1+r)^\beta\p_{r}u}
+\left|2\min\{0,\gamma\}\right|^{1/2}|u(0)|.
\end{align*}

The remaining case $N=2$ requires the use of logarithms.
Moreover, the origin is now a problematic singularity,
which has to be removed from our domain.
Therefore, we may assume $B_1\subset\rN\setminus\Omega$ and $N\geq1$,
having $N=2$ in mind.
We start once more for all $\alpha\in\rz$ with
\begin{align*}
2\intO\frac{r^\alpha}{\ln r}u\p_{r}u\intl
&=\intO\frac{r^\alpha}{\ln r}\p_{r}|u|^2\intl\\
&=-(\alpha+N-1)\intO\frac{r^{\alpha-1}}{\ln r}|u|^2\intl
+\intO\frac{r^{\alpha-1}}{\ln^2 r}|u|^2\intl.
\end{align*}
Now our usual procedure gives for $\gamma\in\rz$ and $\beta:=(\alpha+1)/2\geq0$
\begin{align*}
&\qquad\normom{r^\beta\p_{r}u+\gamma\frac{r^{\beta-1}}{\ln r}u}^2\\
&=\normom{r^\beta\p_{r}u}^2
+|\gamma|^2\normom{\frac{r^{\beta-1}}{\ln r}u}^2
+2\gamma\underbrace{\scpom{r^\beta\p_{r}u}{\frac{r^{\beta-1}}{\ln r}u}}
_{\displaystyle=\intO\frac{r^\alpha}{\ln r}u\p_{r}u\intl}\\
&=\normom{r^\beta\p_{r}u}^2
+\gamma(\gamma+1)\normom{\frac{r^{\beta-1}}{\ln r}u}^2
-\gamma(N+2\beta-2)\normom{\frac{r^{\beta-1}}{\sqrt{\ln r}}u}^2.
\end{align*}
Thus, for $\gamma(N+2\beta-2)\geq0$ we can estimate
$$\normom{r^\beta\p_{r}u+\gamma\frac{r^{\beta-1}}{\ln r}u}^2
\leq\normom{r^\beta\p_{r}u}^2
+\gamma(\gamma-2\beta-N+3)\normom{\frac{r^{\beta-1}}{\ln r}u}^2,$$
which leads to the estimate
$$\normom{r^\beta\p_{r}u+\gamma\frac{r^{\beta-1}}{\ln r}u}^2
\leq\normom{r^\beta\p_{r}u}^2$$
if we set $\gamma:=2\beta+N-3$ with the additional constraint $\gamma(\gamma+1)\geq0$,
i.e. $\gamma\geq0$ or $\gamma\leq-1$.
Finally, again by the triangle inequality
$$|\gamma|\normom{\frac{r^{\beta-1}}{\ln r}u}\leq2\normom{r^\beta\p_{r}u}$$
follows for all $\beta\geq(3-N)/2$ or $\beta\leq(2-N)/2$.
\end{proof}

\begin{acknow}
The authors express their gratitude
to the Department of Mathematical Information Technology
of the University of Jyv\"askyl\"a (Finland)
for financial support.
\end{acknow}

\vspace*{2cm}
\begin{center}
\begin{tabular}{ll}
{\sf Dirk Pauly} & {\sf Sergey Repin}\\
\\
University of Duisburg-Essen & V.A. Steklov Mathematical Institute\\
Faculty for Mathematics & St. Petersburg Branch \\
Campus Essen & \\
Universit\"atsstr. 2 & Fontanka 27 \\
45117 Essen & 191011 St. Petersburg \\
Germany & Russia \\
e-mail: dirk.pauly@uni-due.de & e-mail: repin@pdmi.ras.ru \\
\\
and & \\
\\
University of Jyv\"askyl\"a & University of Jyv\"askyl\"a \\
Faculty of Information Technology & Faculty of Information Technology\\
Department of & Department of \\
Mathematical Information Technology & Mathematical Information Technology \\
P.O. Box 35 (Agora) & P.O. Box 35 (Agora) \\
FI-40014 Jyv\"askyl\"a & FI-40014 Jyv\"askyl\"a \\
Finland & Finland \\
e-mail: dirk.pauly@jyu.fi & e-mail: serepin@jyu.fi
\end{tabular}
\end{center}

\end{document}